\documentclass[11pt]{article}
\usepackage{courier}
\usepackage{amsmath,amsmath,amsthm,amssymb}
\usepackage{fancyhdr}
\usepackage{graphicx}
\usepackage{enumerate}
\usepackage{hyperref}
\usepackage{multicol}

\numberwithin{equation}{section}

\newtheorem{theorem}{Theorem}

\newtheorem {corollary}{Corollary}
\newtheorem {lemma}{Lemma}
\newtheorem {remark}{Remark}

\newcommand{\ds}{\displaystyle}
\newcommand{\ov}{\overline}

\begin{document}
\setlength{\parindent}{4ex} \setlength{\parskip}{1ex}
\setlength{\oddsidemargin}{12mm} \setlength{\evensidemargin}{9mm}
\setlength{\columnsep}{-3 cm} 
%%---------------------------------------------------------------------------------------------------------
%                                          Title
%%---------------------------------------------------------------------------------------------------------
\title{Sharp existence criteria for positive solutions of Hardy--Sobolev type systems}
\author{John Villavert\footnote{email: villavert@math.ou.edu} \\
[0.2cm] {\small Department of Mathematics, University of Oklahoma}\\
{\small Norman, Oklahoma 73019, USA}
 }
\date{}
\maketitle
%% ------------------------------------------------------------------------------------------------------------------------------
%                                      Abstract
%%-------------------------------------------------------------------------------------------------------------------------------
\begin{abstract}
This paper examines systems of poly-harmonic equations of the Hardy--Sobolev type and the closely related weighted systems of integral equations involving Riesz potentials. Namely, it is shown that the two systems are equivalent under some appropriate conditions. Then a sharp criterion for the existence and non-existence of positive solutions is determined for both differential and integral versions of a Hardy--Sobolev type system with variable coefficients. In the constant coefficient case, Liouville type theorems for positive radial solutions are also established using radial decay estimates and Pohozaev type identities in integral form.
\end{abstract}
 {\small \noindent{\bf Keywords:}\, Lane--Emden equations, Hardy--Sobolev inequality; Hardy--Littlewood--Sobolev inequality, fractional integrals; poly-harmonic equations.
\maketitle
%%-------------------------------------------------------------------------------------------------------------------------------
%                                     Introduction
%%-------------------------------------------------------------------------------------------------------------------------------
\section{Introduction and the main results}

In this paper, we examine weighted systems of integral equations involving Riesz potentials of the form
\begin{equation}\label{ie}
    u_{i}(x) = \ds\int_{\mathbb{R}^{n}} \frac{1}{|x-y|^{n-\alpha}|y|^{\sigma_i}}f_{i}(y,u_{1}(y),u_{2}(y),\ldots,u_{L}(y))\,dy,~i=1,2,\ldots, L,
\end{equation}
where $n\geq 3$, $x \in \mathbb{R}^n$, $\alpha \in (0,n)$, $\sigma_i \in [0,\alpha)$ and $f_{i}$ is smooth with respect to its variables. System \eqref{ie} is closely related to the system of pseudo-differential equations involving fractional Laplacians
\begin{equation}\label{pde}
(-\Delta)^{\alpha/2} u_{i} = |x|^{-\sigma_i}f_{i}(x,u_1,u_2,\ldots, u_L) \text{ in } \mathbb{R}^n\backslash\{0\},
\end{equation}
and we shall determine the suitable conditions in which the two systems are equivalent. One reason for examining this general family of systems is to establish and extend results for the system of integral equations of the Hardy--Sobolev type with variable coefficients,
\begin{equation}\label{whls ie}
  \left\{\begin{array}{cl}
    u(x) = c_{1}(x)\ds\int_{\mathbb{R}^{n}} \frac{v(y)^q}{|x-y|^{n-\alpha}|y|^{\sigma_1}}\,dy, \\
    v(x) = c_{2}(x)\ds\int_{\mathbb{R}^{n}} \frac{u(y)^p}{|x-y|^{n-\alpha}|y|^{\sigma_2}}\,dy, \\
  \end{array}
\right.
\end{equation}
including its corresponding system of differential equations
\begin{equation}\label{whls pde}
  \left\{\begin{array}{cl}
    (-\Delta)^{\alpha/2} u(x) = \displaystyle c_{1}(x)\frac{v(x)^q}{|x|^{\sigma_1}} & 	\text{ in } \mathbb{R}^n\backslash\{0\}, \\
    (-\Delta)^{\alpha/2} v(x) = \displaystyle c_{2}(x)\frac{u(x)^p}{|x|^{\sigma_2}} & 	\text{ in } \mathbb{R}^n\backslash\{0\}.
  \end{array}
\right.
\end{equation}
Here the variable coefficients $c_{1}(x)$ and $c_{2}(x)$ are taken to be double bounded functions where a function $c(x)$ is said to be \textit{double bounded} if there exists a $C>0$ such that $1/C \leq c(x) \leq C$ for all $x\in \mathbb{R}^n$. In particular, we establish an optimal or sharp criterion for the existence and non-existence of positive solutions for the integral system \eqref{whls ie}, thereby obtaining an analogous sharp existence result for the differential system \eqref{whls pde}. Our pursuits are inspired by the classical result which states that the scalar equation
\begin{equation}\label{scalar lane-emden pde}
-\Delta u(x) = u(x)^{p},\,~x\,\in \mathbb{R}^n
\end{equation}
admits positive classical solutions whenever $p \geq \frac{n+2}{n-2}$ but has no positive solutions if $p \in (1, \frac{n+2}{n-2})$ (see \cite{CGS89,GS81apriori,GS81}). Basically, this asserts that the exponent $p = \frac{n+2}{n-2}$ is the dividing number that provides a sharp criterion for distinguishing between the existence and non-existence of solutions. Also, the qualitative analysis of such elliptic problems has many applications. The classification of solutions for \eqref{scalar lane-emden pde} has provided an important ingredient in the study of the Yamabe problem and the prescribing scalar curvature problem. In the critical case $p = \frac{n+2}{n-2}$, Gidas, Ni and Nirenberg \cite{GNN81} proved the radial symmetry and monotonicity of positive solutions for \eqref{scalar lane-emden pde} under the additional decay assumption $u = \mathrm{O}(|x|^{2-n})$. In \cite{CGS89}, Caffarelli, Gidas and Spruck removed this decay assumption and proved the same result. Chen and Li \cite{CL91} and Li \cite{Li96} later provided simplified proofs of these results using the method of moving planes and the Kelvin transform. We also mention that the significance of the Hardy--Sobolev systems lies in the fact that they serve as the ``blow-up" equations for many related nonlinear systems of partial differential equations (PDEs). Namely, results on the non-existence of positive entire solutions, often called Liouville type theorems, are essential in deriving a priori estimates and asymptotic and regularity properties of solutions for a large class of nonlinear elliptic equations.

Let us discuss some other notable examples within the family of Hardy--Sobolev type systems. If $\alpha = 2$, $\sigma_{i} \in \mathbb{R}$ and $c_{1},c_{2} \equiv 1$, then system \eqref{whls pde} reduces to the H\'{e}non--Lane--Emden equations, which provides a model for rotating stellar clusters in astrophysics \cite{Henon73} (see also \cite{CMS98,Phan12,PhanSouplet12} and the references therein). In the case where $\alpha \in (0,n)$, $\sigma_{1},\sigma_{2} = 0$ and $c_{1},c_{2} \equiv 1$, system \eqref{whls ie} becomes the well-known Hardy--Littlewood--Sobolev (HLS) system
\begin{equation}\label{hls ie}
  \left\{\begin{array}{cl}
    u(x) = \ds\int_{\mathbb{R}^{n}} \frac{v(y)^q}{|x-y|^{n-\alpha}}\,dy, \\
    v(x) = \ds\int_{\mathbb{R}^{n}} \frac{u(y)^p}{|x-y|^{n-\alpha}}\,dy, \\
  \end{array}
\right.
\end{equation}
which arises as the Euler--Lagrange equations for a functional related to the fundamental Hardy--Littlewood--Sobolev inequality (see \cite{HL,Lieb83,SW58}). Its corresponding system of psuedo-differential equations is the Lane--Emden type HLS system
\begin{equation}\label{hls pde}
  \left\{\begin{array}{cl}
    (-\Delta)^{\alpha/2} u(x) = v(x)^q & \text{ in } \mathbb{R}^n, \\
    (-\Delta)^{\alpha/2} v(x) = u(x)^p &	\text{ in } \mathbb{R}^n.
  \end{array}
\right.
\end{equation}
If $\alpha \in (0,n)$ the critical Sobolev hyperbola, $h(p,q):=\frac{1}{1+q} + \frac{1}{1 + p} = \frac{n-\alpha}{n}$, is known to play a key role in the existence criteria for the unweighted HLS systems. More precisely, the HLS conjecture \cite{Caristi2008} states that \eqref{hls ie} admits no positive classical solution under the subcritical case $h(p,q) > \frac{n-\alpha}{n}$. Even in the case where $\alpha = 2$, the conjecture, better known as the Lane--Emden conjecture, remains a long-standing open problem; however, partial results are known. For instance, the Lane--Emden conjecture has been verified for radial solutions (see \cite{Mitidieri96}; see also \cite{Caristi2008,LGZ06} for when $\alpha > 2$), for dimensions $n\leq 4$ (see \cite{PQS07,SZ96,Souplet09}), and for $n\geq 5$ but for certain regions below the Sobolev hyperbola (see \cite{BM02,DeFF94,Mitidieri96,Souplet09}). On the other hand, the existence of positive solutions in the critical case $h(p,q) = \frac{n-\alpha}{n}$ follows from a variational argument used in finding the best constant in the HLS inequality \cite{Lieb83}. In the non-subcritical case where $\alpha$ is an even integer, the authors in \cite{Li13} and \cite{LGZ06a} obtained the existence of positive solutions for the poly-harmonic HLS system \eqref{hls pde}. In \cite{Villavert:14b}, the author also obtained existence results for the weighted system \eqref{whls pde}.

In the scalar case where $u\equiv v$, $p=q$, $\sigma_1 = \sigma_2 = \sigma$ and $c_{1},c_{2}\equiv 1$, system \eqref{whls ie} reduces to the weighted integral equation
\begin{equation}\label{scalar whls ie}
u(x) = \ds\int_{\mathbb{R}^{n}} \frac{u(y)^p}{|x-y|^{n-\alpha}|y|^{\sigma}}\,dy.
\end{equation}
When $\alpha = 2$ and $p = \frac{n+\alpha - 2\sigma}{n-\alpha}$, this equation is closely related to the Euler--Lagrange equation for the extremal functions of the classical Hardy--Sobolev inequality (see \cite{CKN84,CW01}). In the unweighted case i.e. $\sigma = 0$, the Liouville type properties and the classification of positive solutions for \eqref{scalar whls ie} and its corresponding differential equation,
\begin{equation}\label{scalar hls pde}
(-\Delta)^{\alpha/2}u(x) = u(x)^{p},\,~x\, \in \mathbb{R}^n,
\end{equation}
was established in \cite{CLO05}--\cite{CLO06}, thus extending the aforementioned classical results for equation \eqref{scalar lane-emden pde}. For instance, it was shown that every positive regular solution of the equation with the critical exponent $p = \frac{n+\alpha}{n-\alpha}$ assumes the form
\begin{equation*}
u(x) = c_{n}\left( \frac{\lambda}{\lambda^2 + |x - x_{0}|^2} \right)^{\frac{n-\alpha}{2}}
\end{equation*}
for some positive constants $c_n$ and $\lambda$. We remark that the methods developed in the framework of integral equations have encouraged our studies here since they provide a generalization of the differential systems and generate novel ideas and alternative methods, perhaps leading to new and interesting directions for other related problems. %A recent example of this is the method of moving planes in integral form, introduced in \cite{CLO06}, to derive the symmetry and monotonicity of solutions for a family of integral equations.

We are now ready to state our main results, but prior to doing so, let us first specify the assumptions we place on our general systems. Hereafter, we impose the following conditions on both systems \eqref{ie} and \eqref{pde}: Let $u = (u_1,u_2,\ldots,u_L)$ be a positive solution of either system and suppose $f_{i}(x,u) > 0$ and $f_{i}(x,0) = 0$ for $x \in \mathbb{R}^n$ and $i=1,2,\ldots ,L$. Then we assume that there exist $p_i > 1$, $\sigma \in [0,\alpha)$ and a positive constant $C$ such that
\begin{equation}\label{general condition}
\ds\sum_{i=1}^{L} |x|^{-\sigma_i}f_{i}(x,u_1,u_2,\ldots,u_L)
\geq C|x|^{-\sigma}\sum_{i=1}^{L} u_{i}(x)^{p_i} \,\text{ for }\, x\neq 0.
\end{equation}

The main results of this paper are organized in the following theorems and corollary, the first of which concerns the equivalence between the system of integral equations and the system of differential equations.
\begin{theorem}\label{theorem1}
Let $\alpha = 2k$ be an even integer. Then the system of integral equations \eqref{ie} and the system of differential equations \eqref{pde} are equivalent. That is, every positive solution of \eqref{ie}, multiplied by a suitable constant if necessary, is a positive solution of \eqref{pde}, and vice versa.
\end{theorem}
\begin{remark}
When we refer here to positive solutions of either the integral system \eqref{ie} or the differential system \eqref{pde}, we mean classical solutions $u=(u_1,u_2,\ldots,u_L)$, whose positive components belong in $C^{2k}(\mathbb{R}^{n}\backslash\{0\})\cap C(\mathbb{R}^n)$, satisfying the system pointwise except at the origin.
\end{remark}
The next theorem establishes the necessary and sufficient conditions for the existence of positive solutions to the Hardy--Sobolev type system for some double bounded coefficients.
\begin{theorem}\label{theorem2}
Let $p,q > 0$ and $\sigma_{1},\sigma_{2} \in [0,\alpha)$. Then the system of integral equations \eqref{whls ie} admits a positive solution $(u,v)$ for some double bounded functions $c_{1}(x)$ and $c_{2}(x)$ if and only if $pq > 1$ and
\begin{equation}\label{criteria}
\max\Bigg\{ \frac{\alpha(q+1)-(\sigma_1 +\sigma_2 q)}{pq - 1}, \frac{\alpha(p+1)-(\sigma_2 + \sigma_1 p)}{pq - 1} \Bigg\} < n - \alpha.
\end{equation}
\end{theorem}
Theorems \ref{theorem1} and \ref{theorem2} imply the following.
\begin{corollary}\label{corollary1}
Let $\alpha = 2k$ be an even integer, $p,q > 1$ and $\sigma_{1},\sigma_{2} \in [0,\alpha)$. Then the system of differential equations \eqref{whls pde} admits a positive solution $(u,v)$ for some double bounded functions $c_{1}(x)$ and $c_{2}(x)$ if and only if
\begin{equation*}
\max\Bigg\{ \frac{2k(q+1)-(\sigma_1 +\sigma_2 q)}{pq - 1}, \frac{2k(p+1)-(\sigma_2 + \sigma_1 p)}{pq - 1} \Bigg\} < n - 2k.
\end{equation*}
\end{corollary}
Of course, these results hold for the scalar integral equation and its corresponding differential equation as well and we state them here for completeness sake. The following theorem, however, is essentially contained in \cite{Lei13}.
\begin{theorem}
Let $p>0$ and $\sigma \in [0,\alpha)$. Then there hold the following.
\begin{enumerate}[(i)]
\item The integral equation,
\begin{equation*}
u(x) = c(x)\int_{\mathbb{R}^n} \frac{u(y)^p}{|x-y|^{n-\alpha}|y|^{\sigma}} \,dy,
\end{equation*}
admits a positive solution $u$ for some double bounded function $c(x)$ if and only if $p > \frac{n-\sigma}{n-\alpha}$.

\item Let $p>1$ and $\alpha = 2k$ be an even integer. Then the poly-harmonic equation,
\begin{equation*}
(-\Delta)^k u(x) = c(x)\frac{u(x)^p}{|x|^{\sigma}} \,\text{ in }\, \mathbb{R}^n \backslash \{0\},
\end{equation*}
admits a positive solution $u$ for some double bounded function $c(x)$ if and only if $p > \frac{n-\sigma}{n-2k}$.
\end{enumerate}
\end{theorem}

The remaining two results are Liouville type theorems concerning radially symmetric, decreasing solutions for Hardy--Sobolev type systems with constant coefficients. Here, the radial solutions are understood to belong to $C^{\lfloor \alpha \rfloor}(\mathbb{R}^n \backslash \{0\})\cap C(\mathbb{R}^n)$ where $\lfloor \,\cdot\, \rfloor$ is the greatest integer function. In view of the equivalence between the differential and integral equations, these theorems include the non-existence results for the unweighted Lane--Emden type differential system (see \cite{LGZ06,Mitidieri96}). Moreover, the results of \cite{Villavert:14b} suggest our non-existence results for the Hardy--Sobolev differential equation and system are indeed optimal as far as radial solutions are concerned.

\begin{theorem}\label{theorem4}
Let $\alpha \in [2,n)$ and $\sigma \in (-\infty,\alpha)$. Then the integral equation
\begin{equation}\label{hs eq}
u(x) = \ds\int_{\mathbb{R}^{n}} \frac{u(y)^p}{|x-y|^{n-\alpha}|y|^{\sigma}}\,dy,
\end{equation}
has no positive radial solution if
\begin{equation}\label{subcritical1}
0 < p < \frac{n + \alpha - 2\sigma}{n - \alpha}.
\end{equation}
\end{theorem}

\begin{theorem}\label{theorem5}
Let $p,q > 0$, $\alpha \in [2,n)$ and $\sigma_{1},\sigma_{2}\in (-\infty,\alpha)$. Then the system of integral equations
\begin{equation}\label{hs sys}
  \left\{\begin{array}{cl}
    u(x) = \ds\int_{\mathbb{R}^{n}} \frac{v(y)^q}{|x-y|^{n-\alpha}|y|^{\sigma_1}}\,dy, \\
    v(x) = \ds\int_{\mathbb{R}^{n}} \frac{u(y)^p}{|x-y|^{n-\alpha}|y|^{\sigma_2}}\,dy, \\
  \end{array}
\right.
\end{equation}
has no positive radial solution if $pq \in (0,1]$ or if $pq > 1$ and
\begin{equation}\label{subcritical2}
\frac{n-\sigma_1}{1 + q} + \frac{n-\sigma_2}{1 + p} > n-\alpha.
\end{equation}
\end{theorem}

The rest of this article is organized as follows. In Section \ref{section2}, we begin with several intermediate lemmas and basic results required in our proof of Theorem \ref{theorem1}, which we then provide at the end of the section. Section \ref{section3} starts off with a general non-existence result for the Hardy--Sobolev type integral system followed by the proofs of Theorem \ref{theorem2} and Corollary \ref{corollary1}. The paper concludes with Section \ref{section4}, which provides the proof of both Theorems \ref{theorem4} and \ref{theorem5}.
\begin{remark}
Throughout this paper, we adopt the convention that $C$ represents some constant in the inequalities which may change from line to line. At times, we append subscripts to $C$ to specify its dependence on the subscript parameters.
\end{remark}

\section{Some preliminaries and the proof of Theorem \ref{theorem1}}\label{section2}
A key idea in the proof of Theorem \ref{theorem1} is to multiply the PDEs by a suitable Greene's function, integrate on a ball domain of radius $R$, then send $R\rightarrow \infty$ to show the solutions of the PDE system satisfy the integral system. However, we need to address some technical issues such as the super poly-harmonic property of solutions and some integrability properties of solutions.
\subsection{Preliminaries}
We begin with the super poly-harmonic property for solutions of \eqref{pde}.
\begin{lemma}[Super poly-harmonic property]\label{super polyharmonic}
Let $\alpha = 2k$ be an even integer and suppose that $u=(u_1,u_2,\ldots,u_L)$ is a positive solution of \eqref{pde}. Then
$$(-\Delta)^{j}u_i > 0 \,\text{ for }\, i=1,2,\ldots,L,\, j = 1,2,\ldots,k-1.$$
\end{lemma}
\begin{proof}
We start by proving the result for a single differential inequality then prove it for the general systems, which follows naturally from the scalar case. Suppose $u$ is a positive solution of the differential inequality
\begin{equation}\label{single diff ineq}
(-\Delta)^{k}u(x) \geq C|x|^{-\sigma}u(x)^p, ~ x \in \mathbb{R}^n \backslash \{0\}.
\end{equation}
We assume $p>1$ and $\sigma \in (0,\alpha)$, since the unweighted case when $\sigma = 0$ was treated in \cite{CL13} (see also \cite{CFL13,WeiXu99}). Set
\begin{equation}\label{reduction}
u_{j} := (-\Delta)^{j-1}u \,\text{ for }\, j = 1,2,\ldots, k.
\end{equation}

\medskip

\noindent{\bf Step 1.} First, we show that $u_{k} > 0$. On the contrary, we can have two cases:
\begin{enumerate}[(a)]
\item $u_{k}(x_0) < 0$ at some non-zero point $x_0$;
\item $u_{k} \geq 0$ and $u_{k}(x_{0}) = 0$ at some non-zero point $x_0$.
\end{enumerate}
However, if case (b) holds then $x_0$ is a local minimum, but this contradicts with $-\Delta u_{k} > 0$. Thus, we only consider case (a). In addition, we define the average of a function $v$ on the ball of radius $r$ centered at $\ov{x_0}$ by
\begin{equation*}
\ov{v}(r) = \frac{1}{|\partial B_{r}(\ov{x_0})|}\int_{\partial B_{r}(\ov{x_0})} v(x) \,ds.
\end{equation*}
After a reduction of \eqref{single diff ineq} into a second-order system via \eqref{reduction}, then applying averaging centered at $x_0$ and using the well-known property $\ov{\Delta u} = \Delta\ov{u}$, we get
\begin{equation}\label{reduced system}
\left\{ \begin{aligned}
        &-\Delta \ov{u}_{1} = \ov{u}_{2}, \\
        &-\Delta \ov{u}_{2} = \ov{u}_{3},\\
        & \hspace{1.1cm} \vdots \\
        &-\Delta \ov{u}_{k-1} = \ov{u}_{k},\\
        &-\Delta \ov{u}_{k} > 0.
        \end{aligned} \right.
\end{equation}
From H\"{o}lder's inequality,
\begin{align*}
\overline{u}(r)^{p} = {} & \left(\frac{1}{|\partial B_{r}(x_0)|} \int_{\partial B_{r}(x_0)} u(x) \,ds\right)^p \\
\leq {} & \left(\frac{1}{|\partial B_{r}(x_0)|}\int_{\partial B_{r}(x_0)} |x|^{\frac{\sigma}{p-1}} \,ds \right)^{p-1} \left(\frac{1}{|\partial B_{r}(x_0)|}\int_{\partial B_{r}(x_0)} \frac{u(x)^p}{|x|^{\sigma}} \,ds \right) \\
\leq {} & a_r \cdot \frac{1}{|\partial B_{r}(x_0)|}\int_{\partial B_{r}(x_0)} \frac{u(x)^p}{|x|^{\sigma}} \,ds,
\end{align*}
where $a_r := (|x_0| + r)^{\sigma}$. Thus,
\begin{equation}\label{holder}
-\Delta \ov{u}_{k} = -\ov{\Delta u_{k}} \geq Ca_{r}^{-1} \ov{u}^{p}.
\end{equation}
We claim that after sufficient applications of averaging and re-centers, the resulting solution satisfies an alternating sign property.

\noindent{\bf Claim 1:} After sufficient re-centers of $u_{j}$ and denoting the resulting functions by $\tilde{u}_{j}$, we obtain
\begin{equation}\label{alternating}
\tilde{u}_{k} < 0, \tilde{u}_{k-1} > 0, \tilde{u}_{k-2} < 0,\tilde{u}_{k-3} > 0,\ldots.
\end{equation}
\noindent{\it Proof of Claim 1:} From the $k^{th}$ equation in \eqref{reduced system}, we get
$$\ov{u}_{k} '(r) < 0 \,\text{ and }\, \ov{u}_{k}(r) \leq \ov{u}_{k}(0) < 0 \,\text{ for all }\, r \geq 0. $$
Then the $(k-1)^{th}$ equation implies
\[ -(r^{n-1}\ov{u}_{k-1}'(r))' = r^{n-1}\ov{u}_{k}(r) \leq \ov{u}_{k}(0) r^{n-1} \equiv -c r^{n-1}. \]
Upon integrating we get
\[ \ov{u}'_{k-1}(r) > \frac{c}{n}r \,\text{ and }\, \ov{u}_{k-1}(r) \geq \ov{u}_{k-1}(0) + \frac{c}{2n}r^2 \,\text{ for }\, r \geq 0, \]
so that we can choose $r_1$ sufficiently large such that $\ov{u}_{k-1}(r_1) > 0$. Now, choose some $x_1 \in \partial B_{r_1}(0)$ to be the new point to re-center and apply averaging to the system to get
$$\ov{\ov{u}}_{k-1}(0) > 0,$$
and notice that $\ov{\ov{u}}_{j}$ for $j=1,2,\ldots,k$ satisfy \eqref{reduced system}. As before, we get $$\ov{\ov{u}}'_{k-1} > 0 \,\text{ and }\ \ov{\ov{u}}_{k-1}(r) \geq \ov{\ov{u}}_{k-1}(0) > 0 \,\text{ for all }\, r\geq 0.$$
Thus,
$$\ov{\ov{u}}_{k} < 0 \,\text{ and }\, \ov{\ov{u}}_{k-1} > 0.$$
From the $(k-2)^{th}$ equation, we get
\[ -(r^{n-1}\ov{\ov{u}}_{k-2}'(r))' = r^{n-1}\ov{\ov{u}}_{k-1}(r) \geq \ov{\ov{u}}_{k-1}(0) r^{n-1} \equiv c r^{n-1} \]
and integrating this yields
\[\ov{\ov{u}}'_{k-2}(r) > \frac{c}{n}r \,\text{ and }\, \ov{\ov{u}}_{k-2}(r) \leq \ov{\ov{u}}_{k-2}(0) - \frac{c}{2n}r^2 \,\text{ for }\, r \geq 0. \]
Then choose $r_2$ large enough such that $\ov{\ov{u}}_{k-2}(r_2) < 0$ and pick a point $x_2 \in \partial B_{r_2}(0)$ to be the new point to re-center to get $$\ov{\ov{\ov{u}}}_{k} < 0, \ov{\ov{\ov{u}}}_{k-1} > 0 \,\text{ and }\, \ov{\ov{\ov{u}}}_{k-2} < 0.$$
By repeating this procedure, we arrive at the alternating sign property after sufficient re-centers. This completes the proof of Claim 1.

From Claim 1, we must take $k$ to be an even integer, otherwise we get a contradiction with the positivity of $u$. Henceforth, we assume $k$ is even. Now, define the rescaling of $u$ by
\begin{equation}\label{rescale}
u_{\lambda}(x) := \lambda^{\frac{(2k - \sigma)}{{p-1}}}u(\lambda x)
\end{equation}
and notice that \eqref{single diff ineq} is invariant under this scaling i.e. for any $\lambda > 0$, $u_{\lambda}$ remains a solution of \eqref{single diff ineq}. By \eqref{alternating}, $-\Delta \tilde{u} < 0$, which implies $\tilde{u}'(r) > 0$. Thus,
$$ \tilde{u}(r) \geq \tilde{u}(0) = c_0 > 0. $$
By the scaling invariance, we can choose $u_{\lambda}$ to be as large as needed. Thus, for any $a_0$, we may assume
$$\tilde{u}(r)\geq a_0 \geq a_{0}r^{\beta_0} \,\,\text{ for all }\,\, 0 \leq r \leq 1$$
where the positive number $\beta_0$ is specified below. From \eqref{holder}, we have
\begin{equation*}
-(r^{n-1}\tilde{u}_{k}'(r))' \geq Ca_{0}^{p}r^{n-1 + \beta_{0}p}.
\end{equation*}
Integrating this twice with respect to $r$ and applying \eqref{alternating} yields
\begin{equation*}
\tilde{u}_{k}(r) \leq - \frac{Ca_{0}^{p}r^{2+p\beta_0}}{(n+p\beta_0)(2 + p\beta_0)}.
\end{equation*}
Then
\begin{equation*}
-(r^{n-1}\tilde{u}_{k-1}'(r))' = r^{n-1}\tilde{u}_{k} \leq -\frac{Ca_{0}^{p}r^{n-1 + 2+ \beta_{0}p}}{(n+p\beta_0)(2 + p\beta_0)},
\end{equation*}
and integrating this twice yields
\begin{align*}
\tilde{u}_{k-1}(r) \geq {} & \frac{Ca_{0}^{p} r^{2\cdot 2 + p\beta_0}}{(n+p\beta_0)(n+2 + p\beta_0)(2 + p\beta_0)(2\cdot 2 + p\beta_0)} \\
\geq {} & \frac{Ca_{0}^{p} r^{2\cdot 2 + p\beta_0}}{(n+2 + p\beta_0)^{2}(2\cdot 2 + p\beta_0)^{2}} \\
\geq {} & \frac{Ca_{0}^{p} r^{2\cdot 2 + \beta_{0}p}}{(n + 2\cdot 2 + p\beta_0)^{2\cdot 2}}.
\end{align*}
By continuing in this manner, we obtain
\begin{equation}\label{first step iteration}
\tilde{u}(r) \geq \frac{Ca_{0}^{p} r^{2k + p\beta_0}}{(n + 2k + p\beta_0)^{2k}}
\geq \frac{Ca_{0}^{p} r^{m + p\beta_0}}{(m + p\beta_0)^{m}},~ m = n+2k.
\end{equation}
Choose $\beta_0$ so that $\beta_{0}p \geq m$ and define
$$ \beta_{k+1} = 2\beta_{k}p \,\text{ and }\, a_{k+1} = \frac{a_{k}^{p}}{(2\beta_{k}p)^m}$$
so that $\beta_{k+1} \geq m + \beta_{k}p$ and $\beta_k \rightarrow \infty$ as $k\rightarrow \infty$. Notice that assertion \eqref{first step iteration} shows that $$\tilde{u}(r) \geq a_1 \geq a_{1}r^{\beta_1}\,\,\text{ for all }\,\, 0 \leq r \leq 1.$$
If we apply the previous argument successively, it is not too difficult to show that
\begin{equation}\label{lower bound iteration}
\tilde{u}(r) \geq a_{k}r^{\beta_k}\,\text{ for }\, k=0,1,2,\ldots.
\end{equation}

\noindent{\bf Claim 2:} Choose $l$ so that $l(p-1) > 2$ and further suppose $\beta_0 \geq 2^{1+l+p} p^l$. Then
\begin{equation}\label{claim 2}
a_{k}^{p} \geq (\beta_{k}p)^{m(l+1)},~k=0,1,2,\ldots.
\end{equation}
\noindent{\it Proof of Claim 2:} To prove \eqref{claim 2}, we proceed by induction. The initial case $k=0$ holds immediately since we are free to choose $a_0$ to be as large as needed. Now assume the $k^{th}$ case holds i.e. $$a_{k}^{p} \geq (\beta_{k}p)^{m(l+1)}.$$ Then
\begin{align*}
\frac{a_{k+1}^{p}}{(\beta_{k+1} p)^{m(l+1)}} = {} & \frac{\Big[\frac{a_{k}^{p}}{(2 \beta_{k}p)^{m}}\Big]^{p}}{(p \beta_{k+1})^{m(l+1)}}
\geq \frac{(\beta_{k} p)^{p m(l+1)}}{p^{m(l+1)}(2 \beta_{k} p)^{m p}\beta_{k+1}^{m (l+1)}} \\
= {} & \frac{(\beta_{k} p)^{p m(l+1)}}{p^{m (l+1)}(2\beta_{k} p)^{m p + m (l+1)}} = \frac{\beta_{k}^{ m(l(p-1)-1)}}{2^{m (1+ l + p)} p^{m (l+2 + l(1-p))}} \\
\geq {} & \Bigg[\frac{\beta_{k}}{2^{1+ l + p} p^{l+2 - l(p-1)}}\Bigg]^{m} \geq \Bigg[\frac{\beta_{k}}{2^{1+ l + p} p^l}\Bigg]^{m} \geq \Bigg[\frac{\beta_{0}}{2^{1+ l + p} p^l}\Bigg]^{m} \\
\geq {} & 1.
\end{align*}
Hence, the $(k+1)^{th}$ case holds and this completes the proof of Claim 2.

Consequently, estimates \eqref{lower bound iteration} and \eqref{claim 2} imply
\begin{equation*}
\tilde{u}(1) \geq a_k \geq (\beta_{k}p)^{(m(l+1))/p} \rightarrow \infty \,\text{ as }\, k \rightarrow \infty,
\end{equation*}
which is a contradiction. This proves $u_k > 0$.

\medskip

\noindent{\bf Step 2.} We prove $u_j > 0$ for $j= 2,3,\ldots,k-1$.

On the contrary, assume that $u_{j_0} < 0$ at some point for $j_0 \neq k$. Since $u_k > 0$, we can assume that $u_{j_0 + 1},u_{j_0 + 2},\ldots, u_{k} > 0$. By adopting the same arguments in the proof of Claim 1 with several re-centers and denoting the resulting functions by $\tilde{u}_{j}$, we obtain the alternating sign property $$\tilde{u}_{1} > 0, \tilde{u}_{2}<0,\ldots, \tilde{u}_{j_0 -1} > 0, \tilde{u}_{j_0} < 0 \,\text{ and }\, \tilde{u}_{j_0 + 1},\tilde{u}_{j_0 + 2},\ldots, \tilde{u}_{k} > 0. $$
Thus, by the positivity of $u$ we obtain $-\Delta \tilde{u} < 0$ and $\tilde{u} > 0$, which implies
\begin{equation}\label{positivity}
\tilde{u}(r) \geq \tilde{u}(0) \equiv c_0 > 0 \,\text{ for }\, r \geq 0.
\end{equation}
Then from the $k^{th}$ equation we get $$-(r^{n-1}\tilde{u}_{k}'(r))' \geq Ca^{-1}_{r}\tilde{u}(r)^{p}r^{n-1}\geq Ca^{-1}_{r}r^{n-1},$$
where $a_{r} := (C + r)^{\sigma}$. Integrating this with respect to $r$ yields $$-r\tilde{u}_{k}'(r) \geq Ca^{-1}_{r} r^{2},$$
then applying the standard identity (see Lemma 3.1 in \cite{Mitidieri93})
\begin{equation}\label{standard id}
(n-2)\tilde{u}_{k}(r) + r \tilde{u}_{k}'(r) \geq 0
\end{equation}
yields
\begin{equation}\label{u_k}
\tilde{u}_{k}(r) \geq Ca_{r}^{-1}r^{2}.
\end{equation}
Then integrating the $(k-1)^{th}$ equation, $-(r^{n-1}\tilde{u}_{k-1}'(r))' = r^{n-1}\tilde{u}_{k}(r) $, yields $$-r\tilde{u}_{k-1}'(r) \geq C\tilde{u}_k(r)r^2.$$
Combining this estimate with \eqref{u_k} gives us
$$ \tilde{u}_{k-1}(r) \geq -Cr\tilde{u}_{k-1}'(r) \geq Ca_{r}^{-1}r^{2\cdot 2}.$$
By continuing in this way up to the $(j_{0}+1)^{th}$ equation, we obtain
$$ \tilde{u}_{j_{0}+1}(r) \geq Ca^{-1}_{r}r^{2(k-j_{0})}. $$
From here, we basically mimic the steps at the end of Claim 1; that is, from the $j_{0}^{th}$ equation, $-\Delta \tilde{u}_{j_0}(r) = \tilde{u}_{j_{0}+1}(r) \geq Ca_{r}^{-1}r^{2(k-j_{0})}$, and since $\tilde{u}_{j_{0}} < 0$, we can easily obtain
$$ \tilde{u}_{j_0}(r) \leq -Ca_{r}^{-1}r^{2(k-j_0) + 2}. $$
Likewise, the $(j_{0}-1)^{th}$ equation and $\tilde{u}_{j_{0}-1} > 0$ will imply
$$ \tilde{u}_{j_0 - 1}(r) \geq Ca_{r}^{-1}r^{2(k-j_0) + 2\cdot 2}.  $$
We continue in this way using the alternating sign property with the positivity of $u$ to arrive at the improved lower bound
\begin{equation}\label{lower u}
\tilde{u}(r) \geq Ca_{r}^{-1}r^{2k}.
\end{equation}
Therefore, if $2 - \sigma + (2k - \sigma)p > 0$, we can integrate the resulting inequality from the $k^{th}$ equation,
$$-(r^{n-1}\tilde{u}'_{k}(r))' \geq Ca_{r}^{-1}r^{n-1}\tilde{u}(r)^p \geq Ca_{r}^{-1}r^{n-1}(a^{-1}_{r}r^{2k})^p, $$
to get for some fixed $r_0 > 0$
$$\tilde{u}_{k}(r) \leq \tilde{u}_{k}(0) - Cr^{2-\sigma + (2k -\sigma)p} \,\text{ for }\, r \geq r_0.$$
Then we can choose $R$ suitably large so that $\tilde{u}_{k}(R) \leq 0$, but this is impossible. Otherwise, if $2 - \sigma + (2k - \sigma)p \leq 0$, we can repeat the procedure for obtaining \eqref{lower u} to find a suitably large $m > 1$ depending on $p$ for which $2 - \sigma + m(2k-\sigma)p > 0$ and $\tilde{u}(r) \geq C(a_{r}^{-1}r^{2k})^{m}$. By applying this to the $k^{th}$ equation and integrating, we will arrive at the same contradiction. This completes the proof of Step 2.

\medskip

\noindent{\bf Step 3.} We are ready to prove the theorem for a system of poly-harmonic equations. Assume the contrary; that is, assume for some $i_0,j_0$,
\begin{equation}\label{negative}
(-\Delta)^{j_0}u_{i_0} < 0 \,\text{ at some point.}\,
\end{equation}
Set $v = u_1 + u_2 + \ldots + u_{L}$ and $v_{\epsilon} = u_{i_0} + \epsilon\sum_{i\neq i_0} u_{i}$ for each small $\epsilon > 0$. In view of condition \eqref{general condition}, we can find $p>1$, $\sigma \in [0,\alpha)$ and a suitable constant $C_{\delta} >0$ such that for $v \geq \delta$,
\begin{align*}
(-\Delta)^{k}v_{\epsilon}(x) \geq {} & |x|^{-\sigma_{i_0}}f_{i_0}(x,u_{1},\ldots,u_{L}) + \epsilon\sum_{i\neq i_0} |x|^{-\sigma_i}f_{i}(x,u_{1},\ldots,u_{L}) \\
\geq {} & \epsilon C_{\delta}|x|^{-\sigma}v_{\epsilon}(x)^{p}.
\end{align*}
Here we can apply the same procedures in steps 1 and 2 that we used to derive the super poly-harmonic property for \eqref{single diff ineq} to arrive at
$$ (-\Delta)^{j}v_{\epsilon} > 0, \,\text{ for }\, j=1,2,\ldots,k-1.$$
Then, by choosing $\epsilon$ sufficiently small, we can arrive at a contradiction with \eqref{negative} thereby completing the proof of the theorem.
\end{proof}

Next are two lemmas concerning integrability properties of solutions whose proofs require the following basic results on a fundamental boundary value problem. For each $r> 0$, let $\varphi_{r}(x)$ be the solution of
\begin{align*}
  \left\{\begin{array}{lll}
    (-\Delta)^{k}\varphi(x) = \delta_{x_0}(x) & \text{ in } & B_{r}(x_0), \\
    \varphi = \Delta \varphi = \ldots = \Delta^{k-1}\varphi = 0 & \text{ on } & \partial B_{r}(x_0),
  \end{array}
\right.
\end{align*}
where $\delta_{x_0}(x) = \delta(x-x_0)$ is the centered Dirac delta function. Then
\begin{equation}\label{Hopf}
\frac{\partial}{\partial n}\Big[ (-\Delta)^{k-j}\varphi_{r}(x) \Big] \leq 0 \,\text{ on }\, \partial B_{r}(x_0),\,\text{ for }\, j=1,2,\ldots, k;
\end{equation}
\begin{equation}\label{boundary estimate}
\Big | \frac{\partial}{\partial n}[(-\Delta)^{k-j}\varphi_{r}(x)] \Big |\leq \frac{C}{r^{n+1-2j}} \,\text{ on }\, \partial B_{r}(x_0),\,\text{ for }\, j=1,2,\ldots,k;
\end{equation}
\begin{equation}\label{delta convergence}
\varphi_{r}(x) \rightarrow \frac{c}{|x - x_0|^{n-2k}} \,\text{ as }\, r \rightarrow \infty;
\end{equation}
and
\begin{equation}\label{higher delta convergence}
(-\Delta)^{j-1}\varphi_{r}(x) \rightarrow \frac{c}{|x - x_0|^{n-2(k-j+1)}} \,\text{ as }\, r \rightarrow \infty
\end{equation}
for $j = 2,3,\ldots,k$. Another basic result we often invoke is the following.
\begin{lemma}\label{basic}
If $u\in L^{1}(\mathbb{R}^n)$ and $x_0$ is some point in $\mathbb{R}^n$, then we can find a sequence $\{ r_{\ell} \} \rightarrow \infty$ such that
\begin{equation*}
r_{\ell}\int_{\partial B_{r_{\ell}}(x_0)} |u(x)| \,ds \rightarrow 0.
\end{equation*}
\end{lemma}
For more details on these basic properties and Lemma \ref{basic}, we refer the reader to \cite{CL13,CLO06} and \cite{LL13}.
\begin{lemma}\label{integrable}
Let $\alpha = 2k$ be some positive integer and let $u$ be a solution of \eqref{pde}. Write
$$ u_{ij} = (-\Delta)^{j-1} u_i,~i=1,2,\ldots,L,~j = 1,2,\ldots, k.$$
Then
\begin{equation}\label{integrable 1}
c\int_{\mathbb{R}^n} \frac{1}{|x - x_0|^{n-2k}|x|^{\sigma_i}}f_{i}(x,u_{1}(x),\ldots,u_{L}(x)) \,dx \leq u_{i}(0) < \infty,
\end{equation}
and
\begin{equation}\label{integrable 2}
\int_{\mathbb{R}^n} \frac{u_{ij}(x)}{|x - x_0|^{n+2-2j}} \,dx < \infty \,\text{ for }\, i=1,2,\ldots, L,~j = 2,3,\ldots, k.
\end{equation}
\end{lemma}
\begin{proof}
Multiply both sides of the $i^{th}$ equation,
\begin{equation}\label{ith equation}
(-\Delta)^{k}u_{i}(x) = |x|^{-\sigma_i}f_{i}(x,u_{1}(x),\ldots,u_{L}(x)),
\end{equation}
by $\varphi_{r}(x)$ and integrate on $B_{r}(x_0)$. After successive integration by parts, using the super poly-harmonic property of solutions and \eqref{Hopf}, we get
\begin{align*}
\int_{B_{r}(x_0)} |x|^{-\sigma_i}f_{i}(x,u_{1}(x), {} & \ldots,u_{L}(x))\varphi_{r}(x) \,dx \\
= {} & u_{i}(x_0) + \sum_{j=1}^{k} \int_{\partial B_{r}(x_0)} u_{ij}(x)\frac{\partial}{\partial n}\Big[(-\Delta)^{k-j}\varphi_{r}(x)\Big]\,ds \\
\leq {} & u_{i}(x_0).
\end{align*}
Using this and \eqref{delta convergence}, we can send $r \rightarrow \infty$ to get \eqref{integrable 1}. The proof of \eqref{integrable 2} follows from similar calculations by considering the equation $u_{ij} = (-\Delta)^{j-1}u_{i}$ instead of \eqref{ith equation}, multiplying this equation by $(-\Delta)^{k-j+1}\varphi_r$, integrating on $B_{r}(x_0)$ and then applying \eqref{higher delta convergence}.
\end{proof}
A consequence of Lemma \ref{integrable} is the following.
\begin{lemma}\label{vanishing boundary integral}
There exists a sequence $\{ r_{\ell} \} \rightarrow \infty$ for which
\begin{equation}\label{vbi}
\displaystyle\int_{\partial B_{r_{\ell}}(x_0)} \frac{u_{ij}}{r_{\ell}^{n+1-2j}} \,ds \rightarrow 0 \,\text{ for }\, i=1,2,\ldots,L,\, j=1,2,\ldots,k.
\end{equation}
\end{lemma}
\begin{proof}
For each $i=1,2,\ldots,L$ and $j=2,3,\ldots,k$, the result follows directly from \eqref{integrable 2} and Lemma \ref{basic}, so it remains to verify the lemma for $j=1$. Observe that from \eqref{integrable 1} we can find a sequence $\{ r_{\ell} \} \rightarrow \infty$ for which
\begin{equation}\label{vbu}
\int_{\partial B_{r_{\ell}}(x_0)} \sum_{m=1}^{L} \frac{f_{m}(x,u_{1}(x),\ldots,u_{L}(x))}{|x-x_0|^{n-1-2k}|x|^{\sigma_m}} \,ds \rightarrow 0.
\end{equation}
Then H\"{o}lder's inequality and \eqref{general condition} imply
\begin{align*}
\int_{\partial B_{r}(x_0)} {} & \frac{u_{i}(x)}{|x-x_0|^{n-1}} \,ds \\
\leq {} & \frac{C}{r^{n-1}} r^{(n-1-2k + \sigma + (n-1)(p_i -1))/p_i} \Big(\int_{\partial B_{r}(x_0)} \frac{u_{i}(x)^{p_{i}}}{|x-x_0|^{n-1-2k}|x|^{\sigma}} \,ds \Big)^{\frac{1}{p_i}} \\
\leq {} & \frac{C}{r^{(2k-\sigma)/p_i}}\Big(\int_{\partial B_{r}(x_0)} \frac{u_{i}(x)^{p_i}}{|x-x_0|^{n-1-2k}|x|^{\sigma}} \,ds \Big)^{\frac{1}{p_i}}\\
\leq {} & \frac{C}{r^{(2k-\sigma)/p_i}}\Big( \int_{\partial B_{r}(x_0)} \sum_{m=1}^{L}\frac{f_{m}(x,u_{1}(x),\ldots,u_{L}(x))}{|x-x_0|^{n-1-2k}|x|^{\sigma_m}} \,ds \Big)^{\frac{1}{p_i}}.
\end{align*}
This estimate and \eqref{vbu} imply that we can find a sequence $\{ r_{\ell} \} \rightarrow \infty$ such that
$$\frac{1}{r_{\ell}^{n-1}}\int_{\partial B_{r_{\ell}}(x_0)} u_{i}(x)\,ds \rightarrow 0, \,~\, i=1,2,\ldots,L.$$
This completes the proof.
\end{proof}

\subsection{Proof of Theorem \ref{theorem1}}
We begin by showing the solutions of the PDE system are, up to a multiplicative constant, solutions of the integral system. First, multiply both sides of the $i^{th}$ equation, $$(-\Delta)^{k}u_{i}(x) = |x|^{-\sigma_i}f_{i}(x,u_{1}(x),\ldots,u_{L}(x)),$$ by $\varphi_{r}(x)$ then integrate on $B_{r}(x_0)$. As before, we calculate
\begin{align}\label{main id}
\int_{B_{r}(x_0)} |x|^{-\sigma_i}f_{i}(x,u_{1}(x), {} & \ldots,u_{L}(x)) \varphi_{r}(x) \,dx \notag \\
= {} & u_{i}(x_0) + \sum_{j=1}^{k} \int_{\partial B_{r}(x_0)} u_{ij}(x)\frac{\partial}{\partial n}\Big[(-\Delta)^{k-j}\varphi_{r}(x)\Big]\,ds.
\end{align}
From \eqref{boundary estimate} and Lemma \ref{vanishing boundary integral}, there exists a sequence $r:=r_{\ell} \rightarrow \infty$ in which the boundary integrals in \eqref{main id} vanish. By virtue of \eqref{delta convergence}, \eqref{integrable 1} and the Lebesgue dominated convergence theorem, we conclude that
\begin{equation*}
u_{i}(x_0) = c\int_{\mathbb{R}^n} \frac{1}{|x_0 - y|^{n-2k}|y|^{\sigma_i}}f_{i}(y,u_{1}(y),\ldots,u_{L}(y))\,dy,~i=1,2,\ldots,L.
\end{equation*}

Conversely, showing that the solutions of \eqref{ie} are solutions of \eqref{pde} follows from more elementary arguments. Namely, since $c|x|^{\alpha-n}$ is the fundamental solution of $(-\Delta)^{k}u = \delta_0$, differentiating the integral equations with respect to $x$ and using the convolution properties of the Dirac delta function will show that positive solutions of the integral equations satisfy \eqref{pde}. This completes the proof of the theorem. \qed

\section{Proof of Theorem \ref{theorem2}}\label{section3}
Prior to proving Theorem \ref{theorem2}, we establish a general Liouville type theorem for the Hardy--Sobolev type system with any pair of double bounded coefficients.
\begin{theorem}\label{non-existence}
Let $p,q>0$. Then the integral system \eqref{whls ie} has no positive solution for any double bounded coefficients $c_{1}(x)$ and $c_{2}(x)$ whenever $pq \leq 1$ or if $pq > 1$ and
\begin{equation}\label{non-existence criteria}
\max\Bigg\{ \frac{\alpha(q+1)-(\sigma_1 +\sigma_2 q)}{pq - 1}, \frac{\alpha(p+1)-(\sigma_2 + \sigma_1 p)}{pq - 1} \Bigg\} \geq n - \alpha.
\end{equation}
\end{theorem}

\begin{remark}
We also refer the reader to the earlier papers \cite{Caristi2008,DAmbrosioMitidieri14} for more Liouville type theorems and other interesting results related to the Hardy--Sobolev type systems.
\end{remark}

\begin{proof}[Proof of Theorem \ref{non-existence}]
We proceed by contradiction. That is, assume there is a positive solution $(u,v)$. Let $|x| > R$ for some suitable $R>0$. Note that $|x-y|\leq 2|x|$ for $y\in B_{R}(0)$ and there holds
\begin{equation*}
u(x) \geq C\int_{B_{R}(0)} \frac{v(y)^q}{|x-y|^{n-\alpha}|y|^{\sigma_1}}\,dy \geq \frac{C}{|x|^{a_0}},
\end{equation*}
where $a_0 = n-\alpha$. This estimate and the integral equations imply that for $|x| > R$,
\begin{equation*}
v(x) \geq C\int_{B_{|x|/2}(x)} \frac{|y|^{-p a_0}}{|x-y|^{n-\alpha}|y|^{\sigma_2}}\,dy \geq \frac{C}{|x|^{b_0}},
\end{equation*}
where $b_0 = p a_0 - \alpha + \sigma_2$. Hence, for $|x| > R$
\begin{equation*}
u(x) \geq C\int_{B_{|x|/2}(x)} \frac{|y|^{-q b_0}}{|x-y|^{n-\alpha}|y|^{\sigma_1}}\,dy \geq \frac{C}{|x|^{a_1}},
\end{equation*}
where $a_1 = qb_0 - \alpha + \sigma_1$.
By induction, we obtain
\begin{equation*}
u(x) \geq \frac{C}{|x|^{a_j}} \,\text{ and }\, v(x) \geq \frac{C}{|x|^{b_j}} \,\text{ for }\, |x| > R,
\end{equation*}
where $$a_{k+1} = q b_k - \alpha + \sigma_1 \,\text{ and }\, b_{k} = p a_{k} - \alpha + \sigma_2 \,\text{ for }\, k=0,1,2,\ldots.$$
It follows that
\begin{align}\label{sequence}
a_j = {} & q b_{j-1} - \alpha + \sigma_1 = pq a_{j-1} - \alpha(1+q) \sigma_1 + \sigma_2 q \notag \\
= {} & pq(q b_{j-2} - \alpha + \sigma_1) - \alpha(1+q) + \sigma_1 + \sigma_2 q \notag \\
= {} & (pq)^2 a_{j-2} - (\alpha(1+q) - (\sigma_1 + \sigma_2 q))(pq + 1) \notag \\
= {} & (pq)^3 a_{j-3} - (\alpha(1+q) - (\sigma_1 + \sigma_2 q))((pq)^2 + pq + 1)
 \notag \\
\vdots \notag \\
= {} & (pq)^j a_{0} - (\alpha(1+q) - (\sigma_1 + \sigma_2 q))((pq)^{j-1} + \cdots + (pq)^2 + pq + 1).
\end{align}
There are two cases to consider: when $pq \in (0,1]$ and when $pq > 1$.
\medskip

\noindent{\bf Case 1: } Suppose $pq \in (0,1]$. If $pq = 1$, then \eqref{sequence} implies that
\begin{equation*}
a_{j} = a_{0} - j(\alpha(1+q) - (\sigma_1 + \sigma_2 q)) \rightarrow -\infty
\end{equation*}
and $b_j \rightarrow -\infty$ as $j\rightarrow \infty$. If $pq \in (0,1)$, then \eqref{sequence} implies that
\begin{align*}
a_j = {} & (pq)^{j}(n-\alpha) - (\alpha(1+q) - (\sigma_1 + \sigma_2 q))\frac{(pq)^j - 1}{pq-1} \\
= {} & (pq)^{j}\left(n-\alpha  -\frac{\alpha(1+q) - (\sigma_1 + \sigma_2 q)}{pq-1}\right) + \frac{\alpha(1+q) - (\sigma_1 + \sigma_2 q)}{pq-1} \\
{} & \rightarrow a := -\frac{\alpha(1+q) - (\sigma_1 + \sigma_2 q)}{1-pq} < 0 \,\text{ as }\, j\rightarrow \infty.
\end{align*}
Thus, $b_j \rightarrow a - \alpha + \sigma_2 < 0$. In any case, we can choose a sufficiently large $j_0$ so that $a_{j_0},b_{j_0} < 0$ and
\begin{align*}
u(x) \geq {} & C\ds\int_{\mathbb{R}^n\backslash B_{R}(0)} \frac{v(y)^q}{|x-y|^{n-\alpha}|y|^{\sigma_1}}\,dy
\geq C\int_{\mathbb{R}^{n}\backslash B_{R}(0)} \frac{|y|^{-q b_{j_0}}}{|x-y|^{n-\alpha}|y|^{\sigma_1}} \,dy \\
\geq {} & C\int_{R}^{\infty} r^{\alpha - \sigma - pb_{j_0}} \,\frac{dr}{r} = \infty.
\end{align*}
Thus, $u(x) = \infty$, which is impossible.
\medskip

\noindent{\bf Case 2: } Suppose $pq > 1$.

\noindent 1. Let
\begin{equation*}
n-\alpha < M := \max\Bigg\{ \frac{\alpha(q+1)-(\sigma_1 +\sigma_2 q)}{pq - 1}, \frac{\alpha(p+1)-(\sigma_2 + \sigma_1 p)}{pq - 1} \Bigg\}.
\end{equation*}
First, assume $M = \frac{\alpha(q+1)-(\sigma_1 +\sigma_2 q)}{pq - 1}$. As in Case 1, \eqref{sequence} implies
\begin{equation*}
a_j = (pq)^{j}\left(n-\alpha -\frac{\alpha(1+q) - (\sigma_1 + \sigma_2 q)}{pq-1}\right) + \frac{\alpha(1+q) - (\sigma_1 + \sigma_2 q)}{pq-1} \rightarrow -\infty
\end{equation*}
as $j \rightarrow \infty$. Thus, $u(x) = \infty$, which is impossible. On the other hand, if $pq>1$ and $M = \frac{\alpha(p+1)-(\sigma_2 +\sigma_1 p)}{pq - 1}$, then we can apply the same iteration procedure as above to obtain a contradiction. \\

\noindent 2. Lastly, we consider when $pq>1$ and $M = n-\alpha$. Without loss of generality, we assume $n-\alpha = \frac{\alpha(p+1)-(\sigma_2 +\sigma_1 p)}{pq - 1}$. We see that
\begin{equation*}
u(x) \geq \frac{1}{(R + |x|)^{n-\alpha}}\int_{B_{R}(0)} \frac{v(y)^q}{|y|^{\sigma_1}}\,dy \,\text{ and }\, v(x) \geq \frac{1}{(R + |x|)^{n-\alpha}}\int_{B_{R}(0)} \frac{u(y)^p}{|y|^{\sigma_2}}\,dy.
\end{equation*}
From this, we obtain
\begin{equation*}
\int_{B_{R}(0)} \frac{u(x)^p}{|x|^{\sigma_2}}\,dx \geq \frac{C}{R^{(n-\alpha)p +\sigma_2 - n}}\left(\int_{B_{R}(0)} \frac{v(y)^q}{|y|^{\sigma_1}}\,dy\right)^p,
\end{equation*}
and
\begin{equation*}
\int_{B_{R}(0)} \frac{v(x)^q}{|x|^{\sigma_1}}\,dx \geq \frac{C}{R^{(n-\alpha)q + \sigma_1 - n}}\left(\int_{B_{R}(0)} \frac{u(y)^p}{|y|^{\sigma_2}}\,dy\right)^q.
\end{equation*}
Then
\begin{align}\label{Lp}
\int_{B_{R}(0)} \frac{u(x)^p}{|x|^{\sigma_2}} \,dx \geq {} & \frac{C}{R^{(n-\alpha)p +\sigma_2 - n+ p(q(n-\alpha) +\sigma_1 - n)}}\left(\int_{B_{R}(0)} \frac{u(x)^p}{|x|^{\sigma_2}} \,dy \right)^{pq} \notag \\
\geq {} & C\left(\int_{B_{R}(0)} \frac{u(x)^p}{|x|^{\sigma_2}} \,dy \right)^{pq},
\end{align}
where we are using the fact that
\begin{align*}
(n-\alpha)p +\sigma_2 - n + {} & p(q(n-\alpha) +\sigma_1 - n) \\
= {} & (n-\alpha)p + \sigma_2 - n + pq(n-\alpha) + \sigma_1 p - np \\
= {} & (pq-1)(n-\alpha) - \lbrace \alpha(1+p) - (\sigma_2 + \sigma_1 p)\rbrace \\
%= {} & (pq-1)\Bigg\lbrace (n-\alpha) - \frac{\alpha(1+p) - (\sigma_2 + \sigma_1 p)}{pq-1} \Bigg\rbrace \\
= {} & (pq-1)(n-\alpha - M) \\
= {} & 0.
\end{align*}
Thus, the constant $C>0$ in \eqref{Lp} is independent of $R$. By sending $R\rightarrow \infty$ we get $|x|^{-\sigma_2}u(x)^p \in L^{1}(\mathbb{R}^n)$. By applying similar calculations used in the derivation of \eqref{Lp}, we obtain
\begin{equation*}
\int_{B_{2R}(0)\backslash B_{R}(0)} \frac{u(x)^p}{|x|^{\sigma_2}} \,dx \geq C\left(\int_{B_{R}(0)} \frac{u(x)^p}{|x|^{\sigma_2}} \,dy \right)^{pq},
\end{equation*}
where $C>0$ is independent of $R$. Sending $R\rightarrow \infty$ yields
\begin{equation*}
\int_{\mathbb{R}^n} \frac{u(x)^p}{|x|^{\sigma_2}} \,dx = 0,
\end{equation*}
which implies that $u \equiv 0$ and we arrive at a contradiction. This completes the proof of the theorem.
\end{proof}

\begin{proof}[{\bf Proof of Theorem \ref{theorem2}}]
It suffices to show the existence of positive solutions for system \eqref{whls ie} when $pq > 1$ and \eqref{criteria} holds, since Theorem \ref{non-existence} basically states that this existence result must then be sharp. We show that candidates for solutions are the radial functions
\begin{equation}\label{radial}
u(x) = \frac{1}{(1 + |x|^2)^{\theta_1}} \,\text{ and }\, v(x) = \frac{1}{(1 + |x|^2)^{\theta_2}},
\end{equation}
where
\begin{equation*}
2\theta_1 = \frac{\alpha(1+q) - (\sigma_1 + \sigma_2 q)}{pq-1} \,\text{ and }\, 2\theta_2 = \frac{\alpha(1+p) - (\sigma_2 + \sigma_1 p)}{pq-1}.
\end{equation*}
Notice that \eqref{criteria} implies
\begin{equation}\label{sandwich}
\alpha < 2\theta_1 p + \sigma_2 < n \,\text{ and }\, \alpha < 2\theta_2 q + \sigma_1 < n.
\end{equation}
Note that if $|x| < 2R$ for some suitable $R>0$, $u(x)$ and $v(x)$ are proportional to \\ $\int_{\mathbb{R}^n} \frac{v(y)^q}{|x-y|^{n-\alpha}|y|^{\sigma_1}}\,dy$ and $\int_{\mathbb{R}^n} \frac{u(y)^p}{|x-y|^{n-\alpha}|y|^{\sigma_2}}\,dy$, respectively. Thus, we only consider $|x| > 2R$.

First, it is clear that
\begin{equation*}
\int_{\mathbb{R}^n} \frac{v(y)^q}{|x-y|^{n-\alpha}|y|^{\sigma_1}}\,dy \geq \frac{C}{|x|^{n-\alpha}}\int_{B_{|x|/2}(0)} \frac{v(y)^q}{|y|^{\sigma_1}}\,dy \geq \frac{C}{(1+|x|^2)^{q\theta_{2} + \sigma_{1}/2-\alpha/2}}.
\end{equation*}
This implies
\begin{equation*}
u(x)\left( \int_{\mathbb{R}^n} \frac{v(y)^q}{|x-y|^{n-\alpha}|y|^{\sigma_1}}\,dy\right)^{-1} \leq \frac{C}{(1+|x|^2)^{\theta_{1} - q\theta_{2} - \sigma_{1}/2+\alpha/2}} \leq C,
\end{equation*}
since $\theta_{1} - q\theta_{2} - \sigma_{1}/2+\alpha/2 = 0$. On the other hand,
\begin{align*}
\int_{\mathbb{R}^n} {} & \frac{v(y)^q}{|x-y|^{n-\alpha}|y|^{\sigma_1}}\,dy \\
= {} & \left( \int_{B_{R}(0)} + \int_{B_{|x|/2}(x)} + \int_{B_{R}(0)^{C}\backslash B_{|x|/2}(x)} \right) \frac{v(y)^q}{|x-y|^{n-\alpha}|y|^{\sigma_1}}\,dy \\
:= {} & I_1 + I_2 + I_3,
\end{align*}
where
\begin{align*}
I_1 = \int_{B_{R}(0)} \frac{v(y)^q}{|x-y|^{n-\alpha}|y|^{\sigma_1}}\,dy \leq {} & \frac{C}{(1+|x|^2)^{(n-\alpha)/2}}\int_{B_{R}(0)} \frac{1}{|y|^{2q\theta_2 + \sigma_{1}}}\,dy \\
\leq {} & \frac{C}{(1+|x|^2)^{(n-\alpha)/2}};
\end{align*}
\begin{align*}
I_2 = \int_{B_{|x|/2}(x)} \frac{v(y)^q}{|x-y|^{n-\alpha}|y|^{\sigma_1}}\,dy \leq {} & \frac{C}{(1 + |x|^2)^{q\theta_2 + \sigma_{1}/2}}\int_{B_{|x|/2}(x)} \frac{1}{|x-y|^{n-\alpha}}\,dy \\
\leq {} & \frac{C}{(1 + |x|^2)^{q\theta_2 +\sigma_{1}/2 -\alpha/2}};
\end{align*}
\begin{align*}
I_3 = \int_{B_{R}(0)^{C}\backslash B_{|x|/2}(x)} \frac{v(y)^q}{|x\!-\!y|^{n-\alpha}|y|^{\sigma_1}}\,dy \leq {} & C\int_{B_{R}(0)^{C}\backslash B_{|x|/2}(x)} \frac{1}{|x\!-\!y|^{n-\alpha}|y|^{2q\theta_2 + \sigma_1}} \\
\leq {} & C\int_{|x|/2}^{\infty} r^{n-(n-\alpha + 2q\theta_2 + \sigma_1)}\, \frac{dr}{r} \\
\leq {} & \frac{C}{(1 + |x|^2)^{q\theta_2 +\sigma_{1}/2 -\alpha/2}}.
\end{align*}
These estimates imply
\begin{align*}
\int_{\mathbb{R}^n} {} & \frac{v(y)^q}{|x-y|^{n-\alpha}|y|^{\sigma_1}}\,dy \leq C\Bigg\lbrace \frac{1}{(1 + |x|^2)^{q\theta_2 + \sigma_{1}/2 -\alpha/2}} + \frac{1}{(1+|x|^2)^{(n-\alpha)/2}} \Bigg\rbrace \\
\leq {} & C\Bigg\lbrace \frac{1}{(1 + |x|^2)^{q\theta_2 + \sigma_{1}/2 -\alpha/2 - \theta_1}} + \frac{1}{(1+|x|^2)^{(n-\alpha)/2 -\theta_1}}\Bigg\rbrace\frac{1}{(1+|x|^{2})^{\theta_1}} \\
\leq {} & Cu(x),
\end{align*}
since  \eqref{criteria} implies $(n-\alpha)/2 - \theta_1 > 0$.
Therefore,
\begin{equation*}
\frac{1}{C}\int_{\mathbb{R}^n} \frac{v(y)^q}{|x-y|^{n-\alpha}|y|^{\sigma_1}}\,dy \leq u(x) \leq C\int_{\mathbb{R}^n} \frac{v(y)^q}{|x-y|^{n-\alpha}|y|^{\sigma_1}}\,dy
\end{equation*}
for some suitable constant $C>0$. Hence,
\begin{equation*}
u(x) = c_{1}(x)\int_{\mathbb{R}^n} \frac{v(y)^q}{|x-y|^{n-\alpha}|y|^{\sigma_1}}\,dy
\end{equation*}
for some double bounded function $c_{1}(x)$. Similar calculations show
\begin{equation*}
v(x) = c_{2}(x)\int_{\mathbb{R}^n} \frac{u(y)^p}{|x-y|^{n-\alpha}|y|^{\sigma_2}}\,dy
\end{equation*}
for some double bounded function $c_{2}(x)$. This completes the proof of the theorem.
\end{proof}

\begin{remark}
Instead of solutions with the slow rates, we can also find solutions of the form \eqref{radial} with the fast rate $2\theta_{1} = 2\theta_{2} = n-\alpha$ provided the following stronger conditions hold: $p>\frac{n-\sigma_2}{n-\alpha}$ and $q> \frac{n-\sigma_1}{n-\alpha}$. The proof of this involves similar calculations as before and so we omit the details.
\end{remark}

\begin{proof}[{\bf Proof of Corollary \ref{corollary1}}]
This follows immediately from Theorem \ref{theorem2} provided \eqref{whls ie} and \eqref{whls pde} are equivalent in the following sense.

\medskip

\noindent{\bf Equivalence:} From Theorem \ref{theorem1}, the classical solution of \eqref{whls pde} satisfies
\begin{equation*}
  \left\{\begin{array}{cl}
    u(x) = a_{1}\ds\int_{\mathbb{R}^{n}} \frac{c_{1}(y)v(y)^q}{|x-y|^{n-\alpha}|y|^{\sigma_1}}\,dy, \\
    v(x) = a_{2}\ds\int_{\mathbb{R}^{n}} \frac{c_{2}(y)u(y)^p}{|x-y|^{n-\alpha}|y|^{\sigma_2}}\,dy, \\
  \end{array}
\right.
\end{equation*}
for some positive constants $a_1$ and $a_2$. Clearly,
\begin{equation*}
\frac{1}{C_1}\int_{\mathbb{R}^{n}} \frac{v(y)^q}{|x-y|^{n-\alpha}|y|^{\sigma_1}}\,dy \leq u(x) \leq C_{1}\int_{\mathbb{R}^{n}} \frac{v(y)^q}{|x-y|^{n-\alpha}|y|^{\sigma_1}}\,dy,
\end{equation*}
and
\begin{equation*}
\frac{1}{C_2}\int_{\mathbb{R}^{n}} \frac{u(y)^p}{|x-y|^{n-\alpha}|y|^{\sigma_2}}\,dy \leq v(x) \leq C_{2}\int_{\mathbb{R}^{n}} \frac{u(y)^p}{|x-y|^{n-\alpha}|y|^{\sigma_2}}\,dy,
\end{equation*}
since $c_{1}(x)$ and $c_{2}(x)$ are double bounded. Thus,
\begin{equation}\label{inint}
  \left\{\begin{array}{cl}
    u(x) = a_{1}(x)\ds\int_{\mathbb{R}^{n}} \frac{v(y)^q}{|x-y|^{n-\alpha}|y|^{\sigma_1}}\,dy, \\
    v(x) = a_{2}(x)\ds\int_{\mathbb{R}^{n}} \frac{u(y)^p}{|x-y|^{n-\alpha}|y|^{\sigma_2}}\,dy, \\
  \end{array}
\right.
\end{equation}
for the double bounded functions
\begin{equation*}
a_{1}(x) = u(x)\Bigg[ \ds\int_{\mathbb{R}^{n}} \frac{v(y)^q}{|x-y|^{n-\alpha}|y|^{\sigma_1}}\,dy \Bigg]^{-1}
\end{equation*}
and
\begin{equation*}
a_{2}(x) = v(x)\Bigg[ \int_{\mathbb{R}^{n}} \frac{v(y)^p}{|x-y|^{n-\alpha}|y|^{\sigma_2}}\,dy \Bigg]^{-1}.
\end{equation*}

On the contrary, suppose $u$ and $v$ solve \eqref{inint} for some double bounded functions $a_{1}(x)$ and $a_{2}(x)$. Set $w_{1}(x) = u(x)/a_{1}(x)$ and $w_{2}(x) = v(x)/a_{2}(x)$. Then for any pair of constants $\tilde{c}_{1}$ and $\tilde{c}_{2}$, $(w_{1},w_{2})$ satisfies
\begin{equation*}
  \left\{\begin{array}{cl}
    w_{1}(x) = \tilde{c}_{1}\ds\int_{\mathbb{R}^{n}} \frac{c_{1}(y)w_{2}(y)^q}{|x-y|^{n-\alpha}|y|^{\sigma_1}}\,dy, \\
    w_{2}(x) = \tilde{c}_{2}\ds\int_{\mathbb{R}^{n}} \frac{c_{2}(y)w_{1}(y)^p}{|x-y|^{n-\alpha}|y|^{\sigma_2}}\,dy, \\
  \end{array}
\right.
\end{equation*}
where $c_{1}(x) = a_{2}(x)^p/\tilde{c}_1$ and $c_{2}(x) = a_{1}(x)^p/\tilde{c}_2$. Hence, by the equivalence result of Theorem \ref{theorem1}, we can choose $\tilde{c}_1$ and $\tilde{c}_2$ such that
\begin{equation*}
  \left\{\begin{array}{l}
    (-\Delta)^{\alpha/2} w_{1}(x) = \displaystyle c_{1}(x)\frac{w_{2}(x)^q}{|x|^{\sigma_1}}, \\
    (-\Delta)^{\alpha/2} w_{2}(x) = \displaystyle c_{2}(x)\frac{w_{1}(x)^p}{|x|^{\sigma_2}},
  \end{array}
\right.
\end{equation*}
for double bounded functions $c_{1}(x)$ and $c_{2}(x)$. This completes the proof of the corollary.
\end{proof}

\section{Proof of Theorems \ref{theorem4} and \ref{theorem5}}\label{section4}
\begin{remark} Hereafter, we deliberately neglect the case when $p \in (0,1]$ for \eqref{hs eq} in Theorem \ref{theorem4} and when $pq \in (0,1]$ for \eqref{hs sys} in Theorem \ref{theorem5}, since the non-existence of positive (radial or non-radial) solutions for these follows from Theorem \ref{non-existence} and its scalar counterpart.
\end{remark}
As discussed earlier, these two Liouville type theorems stem from two essential ingredients: the decay estimates for radial solutions and integral forms of Pohozaev type identities. The following lemma addresses the former ingredient.
\begin{lemma}[Decay of radial solutions]\label{decay of u}
Suppose $u=u(r)~(r=|x|)$ is a positive radial solution of \eqref{hs eq}, then there holds the decay property $$u(r) \leq C r^{-\frac{\alpha-\sigma}{p-1}}, \,~\, r > 0.$$
\end{lemma}
\begin{proof}
From the integral equation and the decreasing property of radial solutions, we have
\begin{equation*}
u(r) \geq \int_{B_{r}(0)} \frac{u(y)^p}{|x-y|^{n-\alpha}|y|^{\sigma}}\,dy \geq C r^{n-(n-\alpha +\sigma)}u(r)^{p}
\geq Cr^{\alpha-\sigma}u(r)^p.
\end{equation*}
The result follows by directly solving for $u$ in this inequality.
\end{proof}
As a result of the decay properties of radial solutions, we have the following.
\begin{lemma}\label{finite energy 1}
If $u$ is a positive radial solution of \eqref{hs sys}, then the integrals
\begin{equation*}
\int_{\mathbb{R}^n} \frac{u(x)^{p+1}}{|x|^{\sigma}}\,dx \,\text{ and }\, \int_{\mathbb{R}^n} \frac{u(x)^p}{|x|^{\sigma}}(x\cdot \nabla u(x))\,dx
\end{equation*}
are finite.
\end{lemma}
\begin{proof}
Since $\sigma < n$, the integrals $\int_{B_{R}(0)} \frac{u(x)^{p+1}}{|x|^{\sigma}}\,dx \,\text{ and }\, \int_{B_{R}(0)} \frac{u(x)^p}{|x|^{\sigma}}(x\cdot \nabla u(x))\,dx$ are finite for each $R>0$. Thus, it suffices to show that the following integrals converge:
\begin{equation*}
A_{1}:= \int_{B_{R}(0)^{C}} \frac{u(x)^{p+1}}{|x|^{\sigma}}\,dx \,\text{ and }\, A_{2}:= \int_{B_{R}(0)^{C}} \frac{u(x)^p}{|x|^{\sigma}}(x\cdot \nabla u(x))\,dx.
\end{equation*}
From Lemma \ref{decay of u} and the standard identity we invoked in \eqref{standard id}, we obtain
\begin{align*}
|A_{1}| + |A_{2}| \leq {} & C\int_{R}^{\infty} u(r)^{p}|ru'(r)| r^{n-\sigma}\frac{dr}{r} + C\int_{R}^{\infty} u(r)^{p+1} r^{n-\sigma}\,\frac{dr}{r} \\
\leq {} & C\int_{R}^{\infty} u(r)^{p+1} r^{n-\sigma}\frac{dr}{r} \leq C \int_{R}^{\infty} r^{n-\sigma - (\alpha - \sigma)\frac{p+1}{p-1}} \,\frac{dr}{r},
\end{align*}
where the improper integrals are convergent since $1 < p < \frac{n + \alpha -2\sigma}{n-\alpha}$ implies
\begin{equation*}
n-\sigma - \frac{(\alpha-\sigma)(p+1)}{p-1} = \frac{p(n-\alpha) - (n+\alpha -2\sigma)}{p-1} < 0.
\end{equation*}
\end{proof}

\begin{proof}[\bf Proof of Theorem \ref{theorem4}]
Assume that $u$ is a positive radial solution of the integral equation. We rewrite the integral equation as
\begin{equation*}
u(\lambda x) = \int_{\mathbb{R}^n} \frac{u(y)^{p}}{|\lambda x - y|^{n-\alpha}|y|^{\sigma}}\,dy = \lambda^{\alpha - \sigma}\int_{\mathbb{R}^n} \frac{u(\lambda z)^p}{|x-z|^{n-\alpha}|z|^{\sigma}}\,dz.
\end{equation*}
Differentiating this identity with respect to $\lambda$ on both sides yields
\begin{equation*}
x\cdot \nabla u(\lambda x) = (\alpha - \sigma)\lambda^{\alpha - \sigma -1}\int_{\mathbb{R}^n}\frac{u(\lambda z)^p}{|x-z|^{n-\alpha}|z|^{\sigma}}\,dz + \lambda^{\alpha}\int_{\mathbb{R}^n} \frac{p u(\lambda z)^{p-1}(z\cdot \nabla u)}{|x-z|^{n-\alpha}|z|^{\sigma}}\,dz.
\end{equation*}
Set $\lambda = 1$ to get
\begin{align}\label{xdotwithu}
x\cdot \nabla u(x) = {} & (\alpha - \sigma)\int_{\mathbb{R}^n}\frac{u(z)^p}{|x-z|^{n-\alpha}|z|^{\sigma}}\,dz + \int_{\mathbb{R}^n} \frac{p u(z)^{p-1}(z\cdot \nabla u)}{|x-z|^{n-\alpha}|z|^{\sigma}}\,dz \notag \\
= {} & (\alpha - \sigma)u(x) + \int_{\mathbb{R}^n} \frac{z\cdot \nabla u(z)^{p}}{|x-z|^{n-\alpha}|z|^{\sigma}}\,dz.
\end{align}
To handle the last term in \eqref{xdotwithu}, an integration by parts yields
\begin{align*}
\int_{B_{R}(0)} {} & \frac{z\cdot \nabla u(z)^{p}}{|x-z|^{n-\alpha}|z|^{\sigma}}\,dz = R\int_{\partial B_{R}(0)} \frac{u(z)^{p}}{|x-z|^{n-\alpha}|z|^{\sigma}} \,ds \\
{} & - (n-\sigma)\int_{B_{R}(0)} \frac{u(z)^{p}}{|x-z|^{n-\alpha}|z|^{\sigma}}\,dz - (n-\alpha)\int_{B_{R}(0)} \frac{(z\cdot (x-z))u(z)^{p}}{|x-z|^{n-\alpha + 2}|z|^{\sigma}}\,dz,
\end{align*}
where the boundary integral vanishes as $R\rightarrow \infty$ since $\int_{\mathbb{R}^n} \frac{u(z)^{p}}{|x-z|^{n-\alpha}|z|^{\sigma}} \,dz < \infty.$ With this, we obtain the identity
\begin{equation*}
x\cdot \nabla u(x) = -(n-\alpha)u(x) + (n-\alpha)\int_{\mathbb{R}^n} \frac{(z\cdot (x-z))u(z)^{p}}{|x-z|^{n-\alpha + 2}|z|^{\sigma}}\,dz.
\end{equation*}
If we multiply this by $\frac{u(x)^p}{|x|^{\sigma}}$ and integrate over $\mathbb{R}^n$ we get
\begin{align}\label{int on ball}
\int_{\mathbb{R}^n} {} & \frac{u(x)^p}{|x|^{\sigma}}(x\cdot \nabla u(x))\,dx \notag \\
= {} & -(n-\alpha)\int_{\mathbb{R}^n}\frac{u(x)^{p+1}}{|x|^{\sigma}}\,dx - (n-\alpha)\int_{\mathbb{R}^n}\int_{\mathbb{R}^n} \frac{(z\cdot (x-z))u(x)^{p}u(z)^{p}}{|x-z|^{n-\alpha + 2}|x|^{\sigma}|z|^{\sigma}}\,dz dx.
\end{align}
Noticing that $z\cdot(x-z) + x\cdot(z-x) = -|x-z|^2$, we have
\begin{equation*}
\int_{\mathbb{R}^n}\int_{\mathbb{R}^n} \frac{(z\cdot (x-z))u(x)^{p}u(z)^{p}}{|x-z|^{n-\alpha + 2}|x|^{\sigma}|z|^{\sigma}}\,dz dx
= -\frac{1}{2}\int_{\mathbb{R}^n}\int_{\mathbb{R}^n} \frac{u(x)^{p}u(z)^{p}}{|x-z|^{n-\alpha}|x|^{\sigma}|z|^{\sigma}}\,dz dx.
\end{equation*}
Thus,
\begin{align}\label{pohozaev}
\int_{\mathbb{R}^n} \frac{u(x)^p}{|x|^{\sigma}} {} &  (x\cdot \nabla u(x))\,dx \notag \\
= {} & -(n\!-\!\alpha)\int_{\mathbb{R}^n} \frac{u(x)^{p+1}}{|x|^{\sigma}}\,dx\! -\! (n\!-\!\alpha)\int_{\mathbb{R}^n}\int_{\mathbb{R}^n} \frac{(z\cdot (x\!-\!z)) u(x)^{p}u(z)^{p}}{|x\!-\!z|^{n-\alpha + 2}|x|^{\sigma}|z|^{\sigma}}\,dz dx \notag \\
= {} & -(n-\alpha)\int_{\mathbb{R}^n} \frac{u(x)^{p+1}}{|x|^{\sigma}}\,dx + \frac{n-\alpha}{2}\int_{\mathbb{R}^n} \frac{u(x)^{p+1}}{|x|^{\sigma}}\,dx \notag \\
= {} & -\frac{n-\alpha}{2}\int_{\mathbb{R}^n}\frac{u(x)^{p+1}}{|x|^{\sigma}}\,dx.
\end{align}
In view of Lemma \ref{finite energy 1}, integration by parts yields
\begin{align*}
\int_{B_{R}(0)} \frac{u(x)^p}{|x|^{\sigma}} (x\cdot \nabla u(x))\,dx = {} & \frac{1}{1+p}\int_{B_{R}(0)} \frac{x}{|x|^{\sigma}} \cdot \nabla (u(x)^{p+1})\,dx \notag \\
= {} &\! -\!\frac{(n\!-\!\sigma)}{1\!+\!p}\int_{B_{R}(0)} \frac{u(x)^{p+1}}{|x|^{\sigma}}\,dx
\!+\! \underbrace{R\int_{\partial B_{R}(0)} \frac{u(x)^{p+1}}{|x|^{\sigma}} \,ds}_{ = \mathrm{o}(1) \text{ as } R\rightarrow \infty}.
\end{align*}
Thus,
\begin{align*}
\int_{\mathbb{R}^n} \frac{u(x)^p}{|x|^{\sigma}} (x\cdot \nabla u(x))\,dx = -\frac{(n-\sigma)}{1+p}\int_{\mathbb{R}^n} \frac{u(x)^{p+1}}{|x|^{\sigma}}\,dx.
\end{align*}
Inserting this into \eqref{pohozaev} yields
\begin{equation*}
\Bigg\lbrace \frac{n-\sigma}{1+p} - \frac{n-\alpha}{2} \Bigg\rbrace \int_{\mathbb{R}^n} \frac{u(x)^{p+1}}{|x|^{\sigma}}\,dx
= 0,
\end{equation*}
but this contradicts with \eqref{subcritical1}. This completes the proof of the theorem.
\end{proof}
%%%%%%%%%%%%%%%%%%%%%%%%%%%%%%%%%%%%%%%%%%%%%%%%%%%%%%%%%%%%%%%%%%%%%%%%%%%%%%%%%%%%%%%%%%%%%%%%%%%%%%%%%%%%%%%%%%%%%%%%%%%%%%%%%%%%%%%%
The proof of Theorem \ref{theorem5} is similar to the proof of Theorem \ref{theorem4}. First, we need the following lemma on the decay estimates for radial solutions.
\begin{lemma}\label{decay of uv}
Suppose $(u,v)=(u(r),v(r))$ is a positive radial solution pair of \eqref{hs sys}, then
\begin{equation*}
u(r) \leq C r^{-\frac{\alpha(1+q) - (\sigma_1 + \sigma_{2}q)}{pq-1}} \,\text{ and }\,~ v(r) \leq C r^{-\frac{\alpha(1+p) - (\sigma_2 + \sigma_{1}p)}{pq-1}}, \,~\, r > 0.
\end{equation*}
\end{lemma}
\begin{proof}
From the integral equations and the monotonicity of radial solutions, we obtain
\begin{align*}
u(r) \geq {} & \int_{B_{r}(0)} \frac{v(y)^q}{|x-y|^{n-\alpha}|y|^{\sigma_1}}\,dy \geq Cr^{n-(n-\alpha + \sigma_1)}v(r)^{q}, \\
v(r) \geq {} & \int_{B_{r}(0)} \frac{u(y)^q}{|x-y|^{n-\alpha}|y|^{\sigma_2}}\,dy \geq Cr^{n-(n-\alpha + \sigma_2)}u(r)^{p}.
\end{align*}
Thus, $u(r) \geq C r^{\alpha - \sigma_{1}}v(r)^q$ and $v(r) \geq C r^{\alpha - \sigma_{2}}u(r)^p$, and the result follows by solving for $u$ and $v$ in the estimates.
\end{proof}

As a result of the decay property of radial solutions, we obtain the following.
\begin{lemma}\label{finite energy 2}
Let $(u,v)$ be a positive radial solution of \eqref{hs sys}. Then the following integrals
\begin{align*}
&\int_{\mathbb{R}^n} \frac{u(x)^{p}}{|x|^{\sigma_2}}(x\cdot \nabla u(x)) \,dx,\,~\, \int_{\mathbb{R}^n} \frac{u(x)^{p+1}}{|x|^{\sigma_2}} \,dx, \\
&\int_{\mathbb{R}^n} \frac{v(x)^{q}}{|x|^{\sigma_1}}(x\cdot \nabla v(x)) \,dx, \,\text{ and }\, \int_{\mathbb{R}^n} \frac{v(x)^{q+1}}{|x|^{\sigma_1}} \,dx,
\end{align*}
are finite.
\end{lemma}
\begin{proof}
As in the proof of Lemma \ref{finite energy 1}, it suffices to show the following integrals converge for any finite $R>0$:
\begin{enumerate}[(a)]
\item $B_1 := \ds\int_{B_{R}(0)^{C}} \frac{u(x)^{p+1}}{|x|^{\sigma_2}} \,dx$,
\item $B_2 := \ds\int_{B_{R}(0)^{C}} \frac{v(x)^{q+1}}{|x|^{\sigma_1}} \,dx$,
\item $B_3 := \ds\int_{B_{R}(0)^{C}} \frac{u(x)^{p}}{|x|^{\sigma_2}}(x\cdot \nabla u(x)) \,dx$,
\item $B_4 := \ds\int_{B_{R}(0)^{C}} \frac{v(x)^{q}}{|x|^{\sigma_1}}(x\cdot \nabla v(x)) \,dx$.
\end{enumerate}

From Lemma \ref{decay of uv} and the standard identity we used in \eqref{standard id}, we obtain
\begin{align*}
|B_{1}| + |B_{3}| \leq {} & C\int_{R}^{\infty} r^{n-\sigma_2}u(r)^{p}|ru'(r)|\,\frac{dr}{r} + C\int_{R}^{\infty} r^{n-\sigma_2}u(r)^{p+1} \,\frac{dr}{r} \\
\leq {} & C\int_{R}^{\infty} r^{n-\sigma_2}u(r)^{p+1}\,\frac{dr}{r} \leq C\int_{R}^{\infty} r^{n-\sigma_2 - \frac{\alpha(1+q) - (\sigma_1 + \sigma_{2}q)}{pq-1}(p+1)} \,\frac{dr}{r} \\
\leq {} & C\int_{R}^{\infty} r^{n-\alpha - \frac{\alpha(1+q) - (\sigma_1 + \sigma_{2}q) + \alpha(1+p) - (\sigma_2 + \sigma_{1}p)}{pq-1}}\,\frac{dr}{r} < \infty,
\end{align*}
since the subcritical condition \eqref{subcritical2} is equivalent to
\begin{equation*}
\frac{\alpha(1+q) - (\sigma_1 + \sigma_{2}q) + \alpha(1+p) - (\sigma_2 + \sigma_{1}p)}{pq-1} > n-\alpha.
\end{equation*}
This proves the convergence of the integrals in parts (a) and (c). Parts (b) and (d) follow similarly.
\end{proof}

\begin{proof}[\bf Proof of Theorem \ref{theorem5}]
On the contrary, assume $(u,v)$ is a positive radial solution pair.
As before, we obtain %differentiating this identity with respect to $\lambda$, setting $\lambda = 1$, then applying an integration by parts yield
\begin{equation*}
x\cdot \nabla u(x) = \frac{d}{d\lambda}u(\lambda x)\Big|_{\lambda = 1} = -(n - \alpha)u(x) - (n-\alpha)\int_{\mathbb{R}^n} \frac{(z\cdot (x-z))v(z)^{q}}{|x-z|^{n-\alpha + 2}|z|^{\sigma_1}}\,dz.
\end{equation*}
If we multiply this by $|x|^{-\sigma_2}u(x)^{p}$ and integrate over $\mathbb{R}^n$, we get
\begin{align}\label{int u final}
\int_{\mathbb{R}^n} \frac{u(x)^{p}}{|x|^{\sigma_2}}(x\cdot \nabla u(x))\,dx
= {} & -(n-\alpha)\int_{\mathbb{R}^n} \frac{u(x)^{p+1}}{|x|^{\sigma_2}}\,dx \notag \\
{} & - (n-\alpha)\int_{\mathbb{R}^n}\int_{\mathbb{R}^n} \frac{(z\cdot (x-z))u(x)^{p}v(z)^{q}}{|x-z|^{n-\alpha + 2}|x|^{\sigma_2}|z|^{\sigma_1}}\,dz dx.
\end{align}
Analogous calculations on the second integral equation will lead to
\begin{align}\label{int v final}
\int_{\mathbb{R}^n} \frac{v(x)^{q}}{|x|^{\sigma_1}}(x\cdot \nabla v(x))\,dx
= {} & -(n-\alpha)\int_{\mathbb{R}^n} \frac{v(x)^{q+1}}{|x|^{\sigma_1}}\,dx \notag \\
{} & - (n-\alpha)\int_{\mathbb{R}^n}\int_{\mathbb{R}^n} \frac{(z\cdot (x-z))u(z)^{p}v(x)^{q}}{|x-z|^{n-\alpha + 2}|z|^{\sigma_2}|x|^{\sigma_1}}\,dz dx.
\end{align}
By summing \eqref{int u final} and \eqref{int v final} together and recalling that $z\cdot(x-z) + x\cdot(z-x) = -|x-z|^2$, we get
\begin{equation}\label{pohozaev sys}
\int_{\mathbb{R}^n} \frac{v(x)^{q}}{|x|^{\sigma_1}}(x\cdot \nabla v(x))\,dx + \int_{\mathbb{R}^n} \frac{u(x)^{p}}{|x|^{\sigma_2}}(x\cdot \nabla u(x))\,dx = -(n-\alpha)\int_{\mathbb{R}^n} \frac{v(x)^{q+1}}{|x|^{\sigma_1}}\,dx.
\end{equation}
In view of Lemma \ref{finite energy 2}, integration by parts implies
\begin{align*}
\int_{B_{R}(0)} {} & \frac{v(x)^q}{|x|^{\sigma_1}} (x\cdot \nabla v(x)) \,dx + \int_{B_{R}(0)} \frac{u(x)^p}{|x|^{\sigma_2}} (x\cdot \nabla u(x)) \,dx \notag \\
= {} & -\frac{n-\sigma_1}{1+q}\int_{B_{R}(0)} \frac{v(x)^{q+1}}{|x|^{\sigma_1}}\,dx - \frac{n-\sigma_2}{1+p}\int_{B_{R}(0)} \frac{u(x)^{p+1}}{|x|^{\sigma_2}}\,dx \notag \\
+ {} & \underbrace{\frac{R}{1+q}\int_{\partial B_{R}(0)} \frac{v(x)^{q+1}}{|x|^{\sigma_1}}\,ds}_{ = \mathrm{o}(1) \text{ as } R\rightarrow \infty} + \underbrace{\frac{R}{1+p}\int_{\partial B_{R}(0)} \frac{u(x)^{p+1}}{|x|^{\sigma_2}}\,ds}_{ = \mathrm{o}(1) \text{ as } R\rightarrow \infty}.
\end{align*}
Sending $R\rightarrow \infty$ yields
\begin{align}\label{intparts sys}
\int_{\mathbb{R}^n} \frac{v(x)^q}{|x|^{\sigma_1}} (x\cdot \nabla v(x)) \,dx + {} & \int_{\mathbb{R}^n} \frac{u(x)^p}{|x|^{\sigma_2}} (x\cdot \nabla u(x)) \,dx \notag \\
= {} & -\Bigg\lbrace \frac{n-\sigma_1}{1+q} + \frac{n-\sigma_2}{1+p}\Bigg\rbrace\int_{\mathbb{R}^n} \frac{v(x)^{q+1}}{|x|^{\sigma_1}}\,dx,
\end{align}
where we used the fact that
\begin{equation*}
\int_{\mathbb{R}^n} \frac{u(x)^{p+1}}{|x|^{\sigma_2}}\,dx = \int_{\mathbb{R}^n}\int_{\mathbb{R}^n} \frac{u(x)^{p}v(z)^q}{|x-z|^{n-\alpha}|x|^{\sigma_2}|z|^{\sigma_1}}\,dz dx = \int_{\mathbb{R}^n} \frac{v(x)^{q+1}}{|x|^{\sigma_1}}\,dx.
\end{equation*}
Hence, \eqref{pohozaev sys} and \eqref{intparts sys} imply that
\begin{equation*}
\Bigg\lbrace \frac{n-\sigma_1}{1+q} + \frac{n-\sigma_2}{1+p} - (n-\alpha) \Bigg\rbrace \int_{\mathbb{R}^n} \frac{v(x)^{q+1}}{|x|^{\sigma_1}}\,dx
= 0,
\end{equation*}
but this contradicts with \eqref{subcritical2}. This completes the proof of the theorem.
\end{proof}

\small
\noindent Acknowledgements: This work was completed during a visiting position at the University of Oklahoma. The author would like to thank the university and the Department of Mathematics, especially Professor R\"{u}diger Landes and Professor Meijun Zhu, for their hospitality.

%\bibliographystyle{plain}
%\footnotesize
%\bibliography{refs.bib}
%\nocite{}

\end{document}